\RequirePackage[l2tabu, orthodox]{nag}
\documentclass[a4paper,noamsfonts]{amsart}

\usepackage[T1]{fontenc}
\usepackage[tt/osf=false]{cfr-lm}
\usepackage[british]{babel}

\usepackage{amsmath}
\usepackage{amsthm}
\usepackage{amssymb}
\usepackage{stmaryrd}
\usepackage[foot]{amsaddr}
\usepackage{mathtools}

\usepackage{pgfplots}
\pgfplotsset{compat=1.15}
\usetikzlibrary{cd}
\usetikzlibrary{arrows}
\usetikzlibrary{decorations.markings}
\usepackage{pinlabel}

\usepackage{multirow}
\usepackage[dvipsnames]{xcolor}
\usepackage{graphicx}
\usepackage{booktabs}
\usepackage{todonotes}
\usepackage{pinlabel}

\usepackage{subcaption}

\DeclareCaptionLabelFormat{xx}{#2}
\usepackage{enumitem}
\usepackage[backend=biber, style=numeric, maxnames=8, isbn=false, sortcites=true, backref]{biblatex}
\AtEveryBibitem{\clearlist{language}}
\usepackage{hyperref}
\hypersetup{
    colorlinks,
    linkcolor={red!50!black},
    citecolor={blue!50!black},
    urlcolor={blue!80!black}
}

\usepackage{zref-clever}
\usepackage{microtype}

\zcsetup{noabbrev, cap}
\zcsetup{countertype={thm=theorem}}
\newtheorem{thm}{Theorem}
\numberwithin{thm}{section}
\newtheorem*{thm*}{Theorem}
\newtheorem*{mainthm*}{Main Theorem}
\AddToHook{env/lem/begin}{\zcsetup{countertype={thm=lemma}}}
\newtheorem{lem}[thm]{Lemma}
\AddToHook{env/prp/begin}{\zcsetup{countertype={thm=proposition}}}

\AddToHook{env/cor/begin}{\zcsetup{countertype={thm=corollary}}}
\newtheorem{cor}[thm]{Corollary}
\zcRefTypeSetup{conjecture}{
Name-sg = Conjecture ,
name-sg = conjecture ,
Name-pl = Conjectures ,
name-pl = conjectures ,
}
\AddToHook{env/conj/begin}{\zcsetup{countertype={thm=conjecture}}}

\zcRefTypeSetup{question}{
Name-sg = Question ,
name-sg = question ,
Name-pl = Questions ,
name-pl = questions ,
}
\AddToHook{env/question/begin}{\zcsetup{countertype={thm=question}}}

\theoremstyle{definition}
\AddToHook{env/defn/begin}{\zcsetup{countertype={thm=definition}}}
\newtheorem{defn}[thm]{Definition}
\AddToHook{env/ex/begin}{\zcsetup{countertype={thm=example}}}

\zcRefTypeSetup{cons}{
Name-sg = Construction ,
name-sg = construction ,
Name-pl = Constructions ,
name-pl = constructions ,
}
\AddToHook{env/cons/begin}{\zcsetup{countertype={thm=cons}}}

\zcRefTypeSetup{prob}{
Name-sg = Problem ,
name-sg = problem ,
Name-pl = Problems ,
name-pl = problems ,
}
\AddToHook{env/prob/begin}{\zcsetup{countertype={thm=prob}}}
\theoremstyle{remark}
\AddToHook{env/rem/begin}{\zcsetup{countertype={thm=remark}}}
\newtheorem{rem}[thm]{Remark}
\AddToHook{env/notation/begin}{\zcsetup{countertype={thm=notation}}}

\DeclarePairedDelimiter{\abs}{\lvert}{\rvert}

\newcommand{\df}{\textit}
\newcommand{\union}{\cup}
\newcommand{\inter}{\cap}

\newcommand{\mc}{\mathcal}

\newcommand{\IN}{\mathbb{N}}
\newcommand{\IZ}{\mathbb{Z}}
\newcommand{\IR}{\mathbb{R}}

\newcommand{\IC}{\mathbb{C}}

\newcommand{\IQ}{\mathbb{Q}}
\newcommand{\IH}{\mathbb{H}}
\newcommand{\PSL}{\mathsf{PSL}}

\DeclareMathOperator{\hconv}{conv_{\IH^2}}

\DeclareMathOperator{\Sing}{Sing}

\DeclareMathOperator{\Teich}{Teich}
\DeclareMathOperator{\Stab}{Stab}

\DeclareMathOperator{\Fix}{Fix}
\DeclareMathOperator{\Ax}{Ax}

\DeclareMathOperator{\Mod}{Mod}

\DeclareMathOperator{\Interior}{int}

\begin{document}

\title[Quasi-Fuchsian groups \& $q$-real numbers]{Quasi-Fuchsian groups\\and complex realisations\\of $q$-deformed real numbers}
\author[A. Elzenaar]{Alex Elzenaar}
\address{School of Mathematics, Monash University, Melbourne}
\email{alexander.elzenaar@monash.edu}
\thanks{The author was supported by an Australian Government Research Training Program Scholarship during the period that this work was undertaken.
He thanks Sophie Morier-Genoud and Valentin Ovsienko for helpful discussions. This document was prepared without the use of any generative AI}


\begin{abstract}
  We relate the theory of $q$-rational and $q$-real numbers introduced by Morier-Genoud and Ovsienko to the classical theory of Kleinian groups
  and their Teichm\"uller spaces. This provides a geometric point of view on several recent results about realisations of $q$-rationals for particular
  values of $ q \in \IC $. As an application we resolve a conjecture of Bapat, Becker, and Licata on the topology of the $q$-deformed Farey tessellation.
\end{abstract}

\maketitle

\section{Introduction}
Discrete subgroups of matrix groups (in algebraic group theory) and $q$-series and $q$-analogues (in combinatorics) are well-known to be
connected via the theory of modular and automorphic forms and $ \vartheta$-functions~\cite{ono,farkas}. In this note,
we explain a different connection: we show that the $q$-real numbers introduced by Morier-Genoud and Ovsienko~\cite{moriergenoid20,moriergenoid22}
(see \zcref{sec:qrationals} below) can be interpreted as the limit sets of a family of quasi-Fuchsian groups. The theory of such limit sets dates
back to work of Fricke and Klein in the 1880s and 1890s~\cite{fricke_klein} and an extensive literature was developed in the 1960s by Ahlfors, Bers, Maskit, and others~\cite{ahlfors60,maskit70,bers70}.

The exact family of groups which we consider, the $q$-deformed modular groups $ \PSL(2,\IZ)_q $, is defined in \zcref{sec:qmodulargp} below. In Morier-Genoud,
Ovsienko, and Veselov~\cite{moriergenoud24}, it is argued that the analytic behaviour of the $q$-reals (e.g.\ convergence of the corresponding power series in $q$)
depends on the faithfulness of the canonical representation $ \rho_q : \PSL(2,\IZ) \to \PSL(2,\IZ)_q $. By the standard machinery of quasiconformal
deformation spaces, the connected component of
\begin{displaymath}
  \{ q \in \IC : \rho_q \;\text{is faithful} \}
\end{displaymath}
which contains the identity representation $ q = -1 $ actually \emph{coincides} with the connected component of
\begin{displaymath}
  \{ q \in \IC : \text{the image of}\;\rho_q \;\text{is discrete} \}
\end{displaymath}
containing $ q=-1$. The interior $ \mc{Q} $ of this connected component is described in \zcref{lem:shape_of_q}. The point is that the
statement `$ \PSL(2,\IZ)_q $ is discrete' carries much more geometric information about data like fixed points of elements of $ \PSL(2,\IZ)_q $ than the
statement `$\PSL(2,\IZ)_q \simeq \PSL(2,\IZ) $'---even though the underlying moduli space $ \mc{Q} $ is the same for both statements---and this additional
information is what we rely on to develop the connection with limit sets. Our main result (\zcref{thm:holomotion}) is classical and well-known to people
who work with quasi-Fuchsian groups, but we can use it to deduce many facts about embeddings of the $q$-real numbers into $ \IC $ which seem difficult to prove
directly from the combinatorial viewpoint.

We already used this method in joint work with Gong, Martin, and Schillewaert~\cite{ems24bd} to prove that the Taylor expansions around $0$ of $q$-rational numbers
have radius of convergence at least $ (3-\sqrt{5})/2 $, which is a special case of a conjecture by Leclere, Morier-Genoud, and Ovsienko~\cite[Conjecture~1.1]{leclere24}---more
precisely, Morier-Genoud, Ovsienko, and Veselov~\cite{moriergenoud24} observed that the statement about the radius of convergence is equivalent to a statement about the faithfulness of the Burau
representation $ B_3 $, and we proved this equivalent statement. Recently, Etingoff and Ovsienko~\cite{etingof26} used entirely different techniques to prove that the radii of convergence
of the Taylor expansions of $q$-real numbers satisfy the same bound, which resolves the entire conjecture of \cite{leclere24}.

In this paper, we aim to develop the theory more systematically to make it accessible to people interested in $q$-deformed number systems. To show that the connection
is useful, we use it to prove a conjecture of Bapat, Becker, and Licata~\cite[Conjecture~2.15]{bapat22} (see \zcref{cor:shape_of_qfareytess}). There are also applications
of the theory of $q$-rationals to the study of limit sets: in \zcref{rem:jones_poly} we observe that parts of the limit sets of the groups $ \PSL(2,\IZ)_q $ are obtained as the closure
of the set of values of polynomials obtained from the Jones polynomials of rational links by making all coefficients positive.

\subsection{$q$-rationals and $q$-reals}\label{sec:qrationals}
The $q$-deformed integers
\begin{displaymath}
  [n]_q \coloneq 1 + q + \cdots + q^{n-1} = \frac{1-q^n}{1-q} \;\text{for}\; n \in \IN
\end{displaymath}
were first introduced by Euler~\cite{euler01} and are the genesis of the contemporary study of `$q$-analogues' for
combinatorial objects (see for instance Gasper and Rahman~\cite{gasper}). Morier-Genoud and Ovsienko used the theory
of continued fractions to introduce $q$-analogues for the rational numbers~\cite{moriergenoid20} and proved convergence
results that showed that infinite continued fractions can be used to introduce $q$-analogues for all real numbers~\cite{moriergenoid22}.
Abstractly, recall that $ \PSL(2,\IZ) $ acts transitively in $ \IQ \union \{1/0\} $ and is generated by the matrices
\begin{equation}\label{eq:psl2z}
  R = \begin{bmatrix} 1 & 1 \\ 0 & 1 \end{bmatrix}\; \text{acting as}\; x \mapsto x + 1,\; \text{and}\; S = \begin{bmatrix} 0 & -1 \\ 1 & 0 \end{bmatrix}\; \text{acting as}\; x \mapsto -\frac{1}{x}.
\end{equation}
This transitivity, which is essentially continued fraction decomposition, can be used to show that there is a uniquely defined map
\begin{displaymath}
  [\cdot]_q : \IQ \union \{ 1/0 \} \to \IZ\llbracket q\rrbracket
\end{displaymath}
which satisfies $ [0]_q = 0 $ and the two recursion relations
\begin{displaymath}
  [x+1]_q = q[x]_q + 1 \;\text{and}\; \Big[ -\frac{1}{x} \Big]_q = -\frac{1}{q[x]_q};
\end{displaymath}
if $ f \in \IQ $ then $ [f]_q $ is the \df{$q$-analogue} of $ f $, and the image of $ [\cdot]_q $ is called the set of \df{$q$-rationals}.
Further, if $ (f_i) $ is a sequence of elements of $ \IQ $ converging to some $ \xi \in \IR $, then the $q$-coefficients of $ [f_i]_q $ eventually stabilise and
so can be used to define a power series $ [\xi]_q \in \IZ\llbracket q\rrbracket  $ called the \df{$q$-analogue} of $ \xi $; the image of $ \IR $ under $ \xi \mapsto [\xi]_q $
is the set of \df{$q$-reals}.

\subsection{The $q$-modular group}\label{sec:qmodulargp}
For $ q \in \IC $, let
\begin{equation}\label{eq:psl2zq}
  A = \sqrt{-q}\begin{bmatrix} -q^{-1} & q^{-1} \\ 0 & 1 \end{bmatrix}\;\text{and}\; B = \sqrt{-q}\begin{bmatrix} 1 & 0 \\ 1 & -q^{-1} \end{bmatrix}.
\end{equation}
Define $ \PSL(2,\IZ)_q \coloneq \langle A, B \rangle $. When $ q = -1 $, this is the usual embedding of $ \PSL(2,\IZ) $ into $ \PSL(2,\IC) $. This
group lies on the boundary of the set $ \mc{Q} $ defined by
\begin{displaymath}
  \mc{Q} = \Interior\, \{ q \in \IC : \PSL(2,\IZ)_q\; \text{is discrete and isomorphic to}\; \PSL(2,\IZ) \}.
\end{displaymath}
The following result follows from the Ahlfors--Bers theory of quasiconformal deformation spaces~\cite[\S 4.3]{matsuzaki}. More precisely, each component of the set $ \mc{Q} $
is doubly covered by the $ (2,3)$-Riley slice $ \mc{R}_{2,3} $ studied by Elzenaar, Martin, and Schillewaert~\cite{ems22M} via a map described in Elzenaar, Gong,
Martin, and Schillewaert~\cite[\S 6]{ems24bd}.
\begin{thm}\label{lem:shape_of_q}
  The open set $\mc{Q} \subset \IC $ has two connected components, each of which is conformally equivalent to a punctured disc (the two punctures are at $0$ and $ \infty $).
  For all $ q \in \mc{Q} $, $ \IH^3/\PSL(2,\IZ)_q $ is a hyperbolic orbifold, supported on a $3$-ball, with two ideal orbifold arcs of cone angles $ \pi $ and $ 2\pi/3 $.
  The point $ -1 $ lies on the boundary of $ \mc{Q} $ (in fact, $ -1 $ is the intersection of the closure of the two components), and
  the corresponding orbifold is obtained by pinching an additional simple closed curve on the boundary of the $3$-ball down to length $0$. \qed
\end{thm}
See \zcref{fig:q_def_sp} for a picture of the set $ \mc{Q} $. We will discuss some aspects of \zcref{lem:shape_of_q} in more detail later in this note.
The closure $ \overline{\mc{Q}} $ consists of discrete groups that are all mutually isomorphic to $ \IZ/2\IZ * \IZ/3\IZ$, apart from the two points that come from
filling in the punctures of $ \mc{Q} $. At these two points, some M\"obius transformations in the group converge to constant functions:
for instance, as $ q \to \infty $ the function $ z \mapsto (-z + 1)/q $ represented by $A$ converges uniformly on compact subsets of $ \IC $ to the constant function $0$.

\begin{figure}
  \centering
  \includegraphics[width=.7\textwidth]{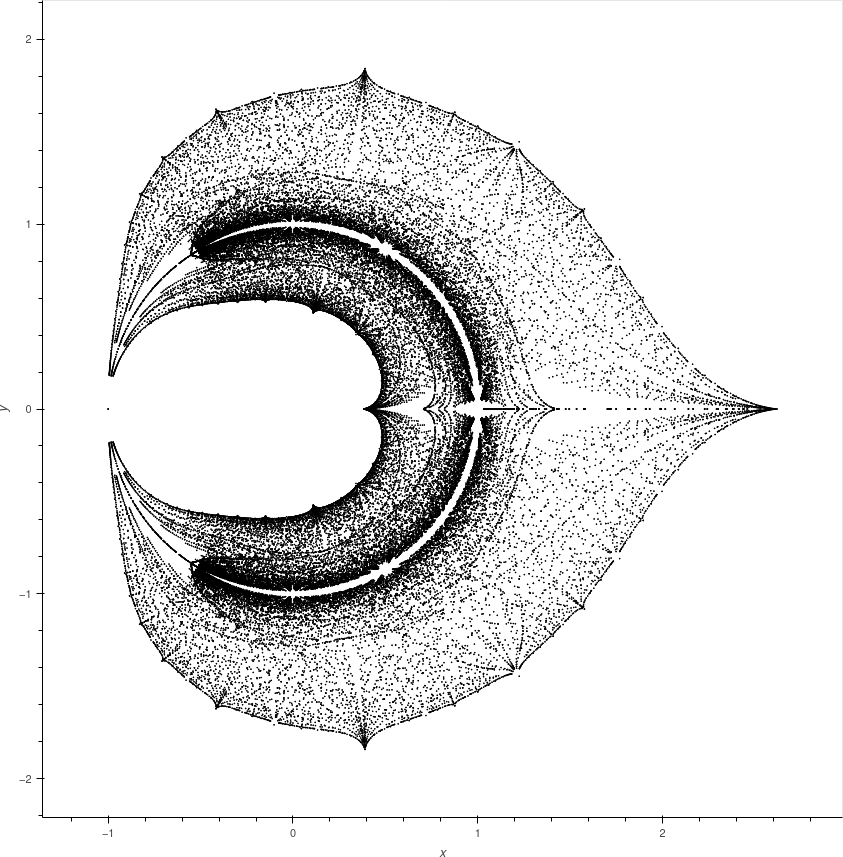}
  \caption{The unshaded region is $\mc{Q}$. The two connected components are exchanged by inversion in the unit circle.\label{fig:q_def_sp}}
\end{figure}

\subsection{Notation and terminology for Kleinian groups}
A \df{Kleinian group} is a discrete subgroup of $ \PSL(2,\IC) $. If $ \Gamma $ is Kleinian, then it admits two natural actions: an action on the Riemann
sphere $ \hat{\IC} = \IC \union \{\infty\} $ by M\"obius transformations, and an action on $ \IH^3 $ by hyperbolic isometries. The action on $ \IH^3 $ is
discontinuous, so $ \IH^3/\Gamma$ is a complete hyperbolic $3$-orbifold (conversely, every complete hyperbolic $3$-orbifold arises in this way). On the other hand,
$ \hat{\IC} $ admits a decomposition into two complementary subsets: the \df{limit set} $ \Lambda(\Gamma) $ of accumulation points of orbits of $ \Gamma $ (which is equal
to the closure of the set of fixed points of infinite-order elements of $ \Gamma $) on which $\Gamma$ acts ergodically, and the \df{ordinary set} $ \Omega(\Gamma) $
which is the maximal open subset of the sphere on which $ \Gamma $ acts discontinuously. The quotient $ \Omega(\Gamma)/\Gamma $ is a (possibly empty, possibly disconnected)
Riemann surface which forms the conformal boundary of $ \IH^3/\Gamma $. When $ \Gamma $ is finitely generated, $\Omega(\Gamma)/\Gamma$ has finitely many components, each
of which has finite genus and finitely many marked points (this is the \df{Ahlfors finiteness theorem}).
For further background see Maskit~\cite{maskit} and Matsuzaki and Taniguchi~\cite{matsuzaki}.

A Kleinian group $\Gamma$ is \df{quasi-Fuchsian} if there exists a Jordan curve $ C \subset \hat{\IC} $ such that $ \Lambda(\Gamma) \subset C $ and such that $ \Gamma $
preserves $C$ and the two components of $ \hat{\IC} \setminus C $. Alternatively, $\Gamma$ preserves $C$ and its orientation. This definition is equivalent to the
existence of a Fuchsian group $ \Gamma_0 $ (i.e.\ a discrete subgroup of $ \PSL(2,\IR) $ acting discontinuously on the upper half-plane $ \IH^2 $) and a quasiconformal
map $ \phi : \hat{\IC} \to \hat{\IC} $ such that $ \Gamma = \phi \Gamma_0 \phi^{-1} $. In particular, the curve $C$ will actually be a \df{quasicircle}, a quasiconformal
deformation of a round circle. This all follows with some minor care from the same arguments as Theorem~2 of Maskit~\cite{maskit70} where the case $ C = \Lambda(\Gamma) $
is considered. An informal but detailed discussion of quasi-Fuchsian groups may be found in Mumford, Series, and Wright~\cite{indras_pearls}.

In \zcref{fig:quasifuchsian} we include a picture of the limit set of a quasi-Fuchsian group, namely $ \PSL(2,\IZ)_q $ for some $ q \in \mc{Q} $.

\begin{figure}
  \centering
  \includegraphics[width=.7\textwidth]{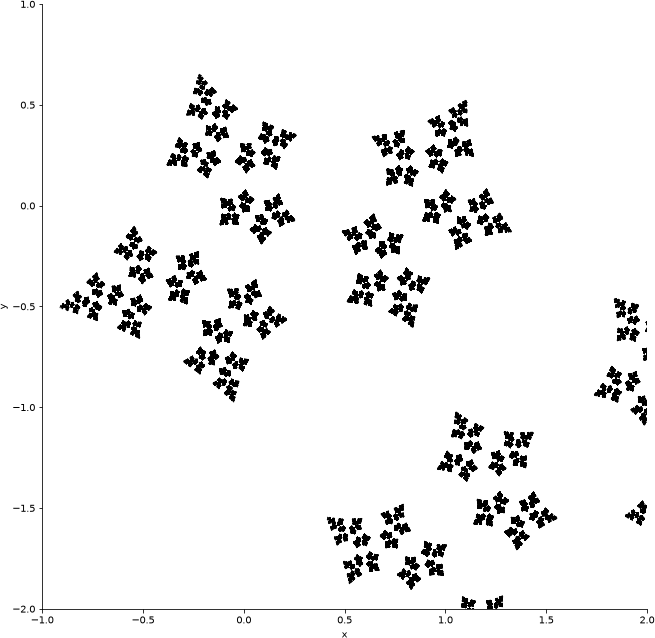}
  \caption{The limit set of $ \PSL(2,\IZ)_q $ where $ q = 0.33+0.39i $. One can alternatively view this as the set of complex realisations of the $q$-reals for $ q = 0.33+0.39i $.\label{fig:quasifuchsian}}
\end{figure}

\section{The classical modular group $\PSL(2,\IZ)$}
Our results for general $q$-rational numbers will be derived from results about deformations of Kleinian groups. In order to do this, we need to
fix a basepoint from which to start the deformations. The classical modular group $ \PSL(2,\IZ) $ is the best choice. In this section, we fix $ q = -1 $, so
\begin{displaymath}
  A = \begin{bmatrix} 1 & -1 \\ 0 & 1 \end{bmatrix}\;\text{and}\; B = \begin{bmatrix} 1 & 0 \\ 1 & 1 \end{bmatrix}
\end{displaymath}
generate $ \PSL(2,\IZ) $. (Compared to the generators in \eqref{eq:psl2z}, $ A = R^{-1} $ and $ B = SR^{-1} S $.)

To understand the action of $ \PSL(2,\IZ) $ geometrically, let us recall some standard concepts from (for instance) Maskit~\cite[\S\S I.C and~II.H]{maskit}. Let $ g \in \PSL(2,\IC) $
be a M\"obius transformation which does not fix $ \infty $. By elementary conformal geometry, $g$ sends the pencil $P$ of circles through $ \xi = g^{-1}(\infty) $ and $ \infty $
to the pencil $Q$ of circles through $ \infty $ and $ \zeta = g(\infty) $. Further, $g$ maps the pencil of circles orthogonal to $P$ (the pencil of circles centred at $ \xi $)
to the pencil of circles orthogonal to $ Q $ (the pencil of circles centred at $\zeta$). By the intermediate value theorem, there is a circle $I^-(g) $ centred at $ \xi $
which is mapped to a circle $ I^+(g) $ centered at $ \zeta $ of the same radius. The circles $ I^+(g) $ and $ I^-(g) $ are called the \df{isometric circles} of $ g $.
\begin{lem}\label{lem:isometric_circles}
  If
  \begin{displaymath}
    g = \begin{bmatrix} a & b \\ c & d \end{bmatrix} \in \PSL(2,\IC)
  \end{displaymath}
  then the isometric circles of $ g $ have radius $ 1/\abs{c} $; the centre of $ I^-(g) $ is $ -d/c $ and the centre of $ I^+(g) $ is $ a/c $.
  The transformation $ g $ acts to send the interior of $ I^-(g) $ onto the exterior of $ I^+(g) $. \qed
\end{lem}
Note that isometric circles are \emph{not} conformal invariants, and depend on the action of $ g$ on Euclidean space. In particular, if $ h \in \PSL(2,\IC) $
then in general it is \emph{not true} that $ I^{\pm}(hgh^{-1}) = h( I^{\pm}(g) ) $. Nevertheless, isometric circles are very important in the study
of Kleinian groups.
\begin{lem}
  Let $G$ be a Kleinian group, and let $ H = \Stab_G(\infty) $. Introduce the temporary notation $ \sigma(g) $ for the centre of $ I^-(g) $ and $ \rho(g) $ for the radius of $ I^-(g) $.
    Then a fundamental polyhedron for $G$ is given by $ A \inter B $ where $A$ is a fundamental polyhedron for $ H $ and where
  \begin{displaymath}
    B = \overline{\bigcap_{g \in G \setminus H}\{ x \in \IH^3 : \abs{x - \sigma(g)} > \rho(g) \}}.
  \end{displaymath}
  This polyhedron is called a \df{Ford domain} for $G$.\qed
\end{lem}

In \zcref{fig:psl2z_ford} we show the isometric circles of $ \PSL(2,\IZ) $ for all words up to length $5$. Actually, a fundamental domain $D$
is bounded by the isometric circles of $S$, which coicide with each other and are the circle of radius $1$ centred at $0$, and the vertical
strip of width $1$ which is a fundamental domain for $ \Stab_{\PSL(2,\IZ)}(\infty) = \langle A \rangle $. This is the usual fundamental
domain for $ \PSL(2,\IZ) $ (see e.g.\ Beardon~\cite[Example~9.4.4]{beardon}). From the fundamental domain, we see immediately that $ \IH^3/\PSL(2,\IZ) $
is an orbifold with the structure claimed in \zcref{lem:shape_of_q} shown in \zcref{fig:psl2z_quotient}: namely, it is an orbifold $O$ such that $ O \setminus \Sing(O) $ is a handlebody of genus $2$,
$ \Sing(O) $ consists of two ideal arcs (one order $2$ and one order $3$), and $ \Omega(\PSL(2,\IZ))/\PSL(2,\IZ) $ is a disjoint union of two thrice-marked
spheres, each marked with $ \{2,3,\infty\} $. The component $ \IH^2/\PSL(2,\IZ) $ is identified with the Riemann moduli space of once-punctured tori, and since
there is an anticonformal reflection mapping $ \IH^2 $ onto the lower half plane $ \overline{\IH^2} $ the two surfaces are indistinguishable other than having
opposite orientations. Transitivity of $ \PSL(2,\IZ) $ on $ \IQ \union \{\infty\} $ is equivalent to the statement that $ \IQ \union \{\infty\} $ is the set of
cusps (i.e.\ parabolic fixed points which are on the boundary of a component of the domain of discontinuity) of $ \PSL(2,\IZ) $.

The set of cusps (the classical rationals) is dense in the limit set of $ \PSL(2,\IZ) $ (the classical reals). Each classical rational lies on the boundary of a unique
copy of the domain $ D $, and the continued fraction decomposition of $ p/q $ gives the sequence of edges of the tessellation $ \PSL(2,\IZ) \cdot D $ of $ \IH^2 $
which leads from $D$ to the copy of $D$ containing $ p/q $. The fact that $ \Lambda(\PSL(2,\IZ)) = \IR \union \{\infty\} $ can be rephrased as stating that every real
number can be approximated by a sequence of copies of $D$ in this tessellation (i.e.\ an infinite walk in the dual graph to the tessellation). We will now state this
a bit more formally.

\begin{figure}
  \centering
  \includegraphics[width=.7\textwidth]{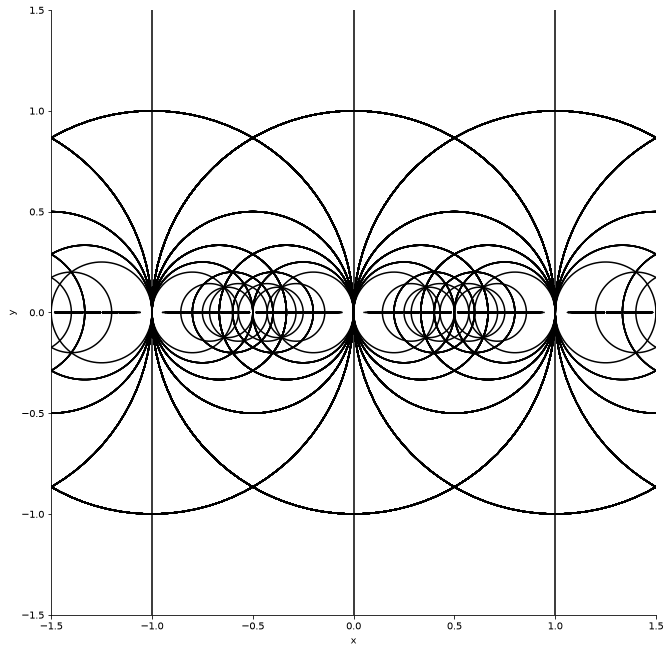}
  \caption{Isometric circles for words in $ \PSL(2,\IZ) $ of length up to $5$ in the generators $A$ and $B$.\label{fig:psl2z_ford}}
\end{figure}

\begin{figure}
  \centering
  \labellist
  \small\hair 2pt
  \pinlabel {$2$} [t] at 208 129
  \pinlabel {$3$} [b] at 232 16
  \pinlabel {$\infty$} [br] at 214 52
  \pinlabel {$\IH^2/\PSL(2,\IZ)$} at 60 63
  \pinlabel {$\overline{\IH^2}/\PSL(2,\IZ)$} at 393 80
  \endlabellist
  \includegraphics[width=.8\textwidth]{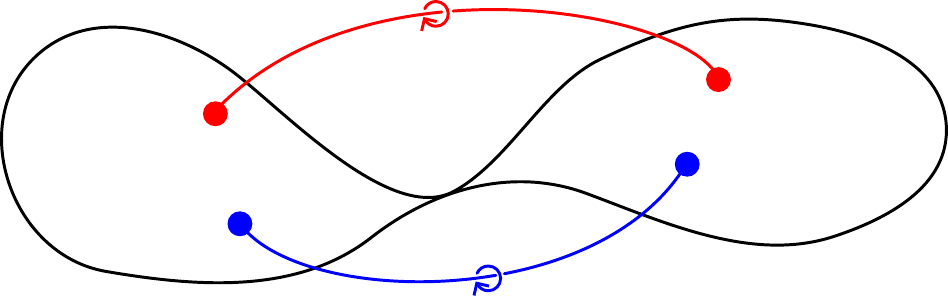}
  \caption{The orbifold $ \IH^3/\PSL(2,\IZ) $. The two components of $ \Omega(\PSL(2,\IZ)) $ are the upper halfplane $ \IH^2 $ and the lower halfplane $ \overline{\IH^2} $; each descends
  to a surface on the conformal boundary of the orbifold.\label{fig:psl2z_quotient}}
\end{figure}
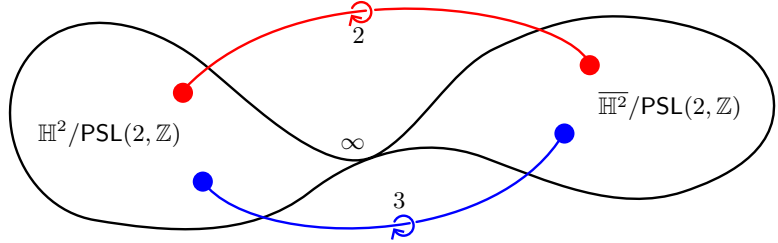

If $ r/s \in \IQ \union \{\infty\} $ then it admits exactly two continued fraction expansions,
\begin{displaymath}
  \frac{r}{s} = [a_0; a_1, a_2, \cdots, a_n] \coloneq a_0 + \cfrac{1}{a_1 + \cfrac{1}{a_2 + \cfrac{1}{\ddots + \cfrac{1}{a_n}}}}
\end{displaymath}
one of odd length and one of even length.

\begin{defn}\label{defn:rational_word}
  The \df{rational word} of slope $ r/s $ is defined as follows. Suppose that $ r/s \in \IQ_{\geq 0} \union \{\infty\} $ has continued fraction expansion $ [a_0;b_1,a_1,\ldots,b_n,a_n] $ (of
  odd length with no negative coefficients). Then
  \begin{displaymath}
    W_{r/s} = A^{-a_0} B^{b_1} A^{-a_1} \cdots B^{b_n} A^{-a_n}.
  \end{displaymath}
  For $ r/s \in \IQ_{<0} $, we set $ W_{r/s} = S W_{-s/r} $ where $ S = ABA $ is the involution $ z \mapsto -1/z $.
\end{defn}
\begin{lem}
  $ W_{r/s}(\infty) = r/s $.
\end{lem}
\begin{proof}
  The Euclidean algorithm (e.g.\ Burde and Zieschang~\cite[\S 12.11]{burde3e}) shows that
  \begin{displaymath}
    W_{r/s} = \begin{bmatrix} r & \ast \\ s & \ast \end{bmatrix}.
  \end{displaymath}
  The result then follows from \zcref{lem:isometric_circles}.
\end{proof}
One can view the rational words as modelling a path in the Farey tree from $ \infty $ to $ p/q $. This tree can be seen as the dual graph to a tessellation induced by $ \PSL(2,\IZ) $. Indeed, consider the $3$-fold
cover of $ \IH^2/\PSL(2,\IZ) $ with fundamental domain $F$ shown in \zcref{fig:farey_tricover}. This domain $F$ is exactly the simplex $ (0,1,\infty) $ in the Farey tessellation.
The group generated by $A$ and $B$ does not act to tile the plane with copies of $F$ since there are overlaps, but $ \{W_{r/s} F : r/s \in \IQ \} $ \emph{is} a tiling of the plane.

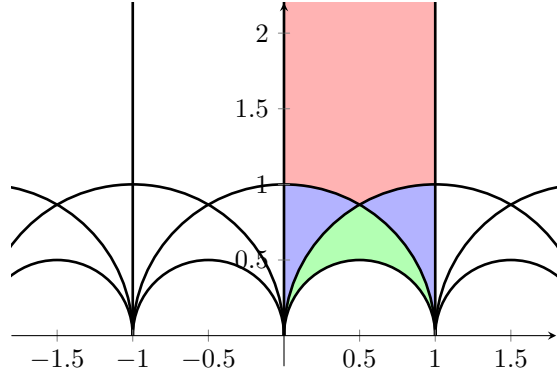
\begin{figure}
  \centering
  \begin{tikzpicture}[line cap=round,line join=round,>=triangle 45,x=2.0cm,y=2.0cm]
  \begin{axis}[
  x=2.0cm,y=2.0cm,
  axis lines=middle,
  xmin=-1.8,
  xmax=1.8,
  ymin=-0.2,
  ymax=2.2,
  xtick={-1.5,-1.0,...,1.5},
  ytick={-0.0,0.5,...,5.0},axis on top]
  \clip(-1.8,0) rectangle (1.8,5.);
  \draw[fill=red!30,line width=0]
    ([shift={(90:2cm)}]0,0) arc (90:60:2cm)--
    ([shift={(120:2cm)}]1,0) arc (120:90:2cm)--(1,2.5)--(0,2.5)
    -- cycle;
  \draw[fill=blue!30,line width=0]
    ([shift={(90:2cm)}]0,0) arc (90:60:2cm)--
    ([shift={(120:2cm)}]1,0) arc (120:180:2cm)
    -- cycle;
  \draw[fill=blue!30,line width=0]
    ([shift={(0:2cm)}]0,0) arc (0:60:2cm)--
    ([shift={(120:2cm)}]1,0) arc (120:90:2cm)
    -- cycle;
  \draw[fill=green!30,line width=0]
    ([shift={(0:2cm)}]0,0) arc (0:60:2cm)--
    ([shift={(120:2cm)}]1,0) arc (120:180:2cm)--
    ([shift={(180:1cm)}]0.5,0) arc (180:0:1cm);
  \draw [line width=1.pt] (0.,0.) circle (2.0cm);
  \draw [line width=1.pt] (1.,0.) circle (2.0cm);
  \draw [line width=1.pt] (2.,0.) circle (2.0cm);
  \draw [line width=1.pt] (-2.,0.) circle (2.0cm);
  \draw [line width=1.pt] (-1.,0.) circle (2.0cm);
  \draw [line width=1.pt] (0.,-0.2) -- (0.,5.);
  \draw [line width=1.pt] (1.,-0.2) -- (1.,5.);
  \draw [line width=1.pt] (2.,-0.2) -- (2.,5.);
  \draw [line width=1.pt] (-2.,-0.2) -- (-2.,5.);
  \draw [line width=1.pt] (-1.,-0.2) -- (-1.,5.);
  \draw [line width=1.pt] (-0.5,0.) circle (1.0cm);
  \draw [line width=1.pt] (0.5,0.) circle (1.0cm);
  \draw [line width=1.pt] (1.5,0.) circle (1.0cm);
  \draw [line width=1.pt] (-1.5,0.) circle (1.0cm);
  \end{axis}
  \end{tikzpicture}
  \caption{A $2$-simplex in the Farey triangulation is exactly a fundamental domain for a $3$-fold cover of $ \IH^2/\PSL(2,\IZ) $.\label{fig:farey_tricover}}
\end{figure}

Further, if $ \xi \in \IR \setminus \IQ $ then we may find a sequence $ \alpha_i $ of rational numbers such that $ \{ \alpha_{i-2}, \alpha_{i-1}, \alpha_i \} $ are
a Farey triple for all $ i $ (see e.g.~\cite[Proposition~7.1]{elzenaar26}). It follows that the sequence $ W_{\alpha_i}(F) $ converges in the Hausdorff topology
to the point $ \xi $ (since $W_{\alpha_i} (F)$ is a convex hull in $ \IH^2 $ of the three points $\alpha_{i-2}, \alpha_{i-1}, \alpha_i$ and these all converge to $ \xi $).
Thus in some sense the rational words $ W_{r/s} $ converge at the ends of the Farey tree to the constant functions which take irrational values.

\section{$q$-rational numbers}
We now give an alternative definition of the $q$-rational numbers due to Bapat, Becker, and Licata~\cite[Definition~2.6]{bapat22}, which extends that given by Morier-Genoud
and Ovsienko~\cite{moriergenoid22,moriergenoid20} to assign to each $ p/q \in \IQ $ a \emph{pair} of $q$-deformations termed the \df{left} and \df{right} $q$-rationals.
By slight abuse of notation, in this section we will $ W_{r/s} $ for the rational words defined in \zcref{defn:rational_word} but now where $ A $ and $ B $ are the matrices
defined in \eqref{eq:psl2zq}; thus $ W_{r/s} $ can be viewed as a matrix with coefficients in $ \IC\llbracket q^{-1/2} \rrbracket $. The substitution of the variable $q$
with an actual number $ q \in \IC $ will be called a \df{realisation} (of the matrix or $q$-deformed number) given by the value $q$.

\begin{defn}
  For $ q \in \mc{Q}$, the $q$-deformations of the rational number $r/s$ are the two images of $ \Fix A = \{ \infty, 1/(1+q) \} $ under
  the rational word $ W_{r/s}$. The \df{left $q$-rational}, $ [r/s]_q^\flat $, is $ W_{r/s}(1/(1+q)) $. The \df{right $q$-rational}, $ [r/s]_q^\sharp $,
  is $ W_{r/s}(\infty) $ and agrees with the original $q$-rational described in the introduction.
\end{defn}
The limits of the right $q$-rationals are the \df{$q$-irrationals} as defined by Morier-Genoud and Ovsienko. It will follow from \zcref{cor:limset} below
that one does not obtain any `new' $q$-deformed numbers as limits of the left $q$-rationals. (Compare with Theorem~2.12 of~\cite{bapat22}.)
\begin{rem}
  Our definition agrees with Definition~2.6 of \cite{bapat22} up to the substitution $ q \mapsto -q $. This is due to a slightly different normalisation for $ \PSL(2,\IZ)_q $,
  compare \eqref{eq:psl2zq} with \cite[Definition~2.3]{bapat22}.
\end{rem}

When $ q \in (-\infty,-1) $, the two points $ [r/s]_q^\flat $ and $ [r/s]_q^\sharp $ lie on $ \IR $, and Bapat, Becker, and Licata~\cite{bapat22} study
the geodesic in $ \IH^2 $ which joins them. They observe that the deformed Farey tessellation studied in Morier-Genoud and Ovsienko~\cite{moriergenoid20}
converges onto these geodesics. We will now show that observations like this can be proved using the machinery of Kleinian groups, and in the process we
will give a conceptual picture of the realisations of the $q$-reals in $ \IC $ for every $ q \in \mc{Q} $.

The following theorem is really a different view of \zcref{lem:shape_of_q}.
\begin{thm}\label{thm:holomotion}
  Let $ \alpha : [0,1) \to \mc{Q} $ be an analytic path. Let $ \Gamma_t = \PSL(2,\IZ)_{\alpha(t)} $.
  \begin{enumerate}
    \item The oriented quasicircle $ C_t $ preserved by $ \Gamma_t $ moves holomorphically as a subset of the Riemann sphere, and this motion induces
          an injective order-preserving map
          \begin{displaymath}
            \lambda_t : \Lambda(\Gamma_0) \to \Lambda(\Gamma_t).
          \end{displaymath}
    \item If $ \lim_{t\to 1} \alpha(t) = -1 $ (so $ \Gamma_t \to \PSL(2,\IZ) $ in the algebraic convergence topology), then the maps $ \lambda_t $
          converge to a map $ \lambda_1 $ that is injective on all points of $ \Lambda(\Gamma_0) $ except for fixed points of elements $ g A g^{-1} $ ($ g\in \Gamma_0 $):
          that is, for $ x,y \in C_0 $,
          \begin{displaymath}
            \lambda_1(x) = \lambda_1(y) \iff \{x,y\} = \Fix(g A g^{-1}) \;\text{for some}\;g \in \Gamma_0.
          \end{displaymath}
          In addition, $ \lambda_1 $ is still order-preserving on the topological curve $ C_1 = \lim_{t\to 1} C_t $.
  \end{enumerate}
\end{thm}
\begin{proof}
  Since $ \alpha$ is analytic it can be extended to a holomorphic map on a small neighbourhood of $ [0,1) $ in $ \IC $. The measurable Riemann mapping theorem gives holomorphic
  dependence of the quasiconformal maps conjugating $ \Gamma_t $ to $ \Gamma_0 $ on the parameter set $ \mc{Q} $~\cite{ahlfors60} (see also~\cite[\S 5.7]{astala}).
  Thus we obtain an equivariant holomorphic motion of the curves $ C_t $, which is clearly order-preserving. This shows part (1).

  To see part (2), we need only check that no other limit points of $ \Gamma_t $ collide as $ t \to 1 $. This result dates back to Bers~\cite{bers70}, but we will
  give a conceptual geometric argument. Suppose that there exist $ \xi, \zeta \in \Lambda(\Gamma_0) $ such that $ \xi \neq \zeta $ but $ \lambda_1(\xi) = \lambda_1(\zeta) $.
  Write $ \xi_t \coloneq \lambda_t(\xi) $ and $ \zeta_t \coloneq \lambda_t(\zeta) $. Let $ D^{\pm}_t $ be consistent choices of the two complements of the orientable Jordan curve
  preserved by $ \Gamma_t $, and place the Poincar\'e metric on each~\cite[Theorem 1-9]{ahlfors}.  Consider the respective geodesics $ \gamma_t^{\pm} $ in $ D^{\pm}_t $ which
  join the two boundary points $ \xi_t $ and $ \zeta_t $. As $ t \to 1 $, one of $ \gamma^+_t $ or $ \gamma^-_t $ must degenerate to a point. Let $ \gamma_t $ and $ D_t $ be
  this geodesic and the corresponding disc. We have two cases.
  \begin{itemize}
    \item \textit{The geodesic $ \gamma_t $ descends to a closed geodesic loop on $ D_t/\Gamma_t $.} In this case, $ \gamma_t $ joins two fixed points of the same element $ g_t $ (namely, the primitive element
    of $ \Gamma_t \simeq \pi_1(D_t/\Gamma_t)$ representing the closed loop). This means that as $ t\to1 $, the fixed points of $ g_t $ collide. Hence $ g_1 $ is parabolic. But the only parabolic
    elements of $ \PSL(2,\IZ) $ are conjugates of $A$.
    \item \textit{The geodesic $ \gamma_t $ descends to a non-compact geodesic on $ D_t/\Gamma_t $.} This means that as $ t \to 1 $, there is a non-compact geodesic which is pinched to a point. But
    $D_t/\Gamma_t$ (for $ 0 \leq t < 1$) is a disc with two marked points, and the deformation as $ t\to -1 $ simply corresponds to shrinking the length of the boundary of the disc to $0$ in the induced
    hyperbolic metric; under this process, all non-compact geodesics limit onto non-compact geodesics and not points, giving a contradiction. \qedhere
  \end{itemize}
\end{proof}

As an immediate corollary of part (2) of \zcref{thm:holomotion} combined with the study of the limit set of $\PSL(2,\IZ) $ in the previous section, we see:
\begin{cor}\label{cor:limset}
  Fix $ q \in \mc{Q} $. The limit set of $ \PSL(2,\IZ)_q $ is equal to the set of complex realisations of the $q$-reals given by this particular $q$. \qed
\end{cor}

\begin{rem}[Jones polynomials of rational links]\label{rem:jones_poly}
  Via \zcref{cor:limset}, Theorem~A.3 of Bapat, Becker, and Licata~\cite{bapat22} gives an interesting new interpretation of the limit set $ \Lambda(\PSL(2,\IZ)_q) $: for $ r/s \in (1,\infty) $,
  let $ V_{r/s}(q) $ denote the Jones polynomial of the $2$-bridge link of slope $ r/s $ and let $ \abs{V_{r/s}(q)} $ denote the polynomial obtained by making each coefficient positive.
  Then for each $q$, $ \abs{V_{r/s}(-q)} $ is dense in one of the two halves of the limit set of $ \PSL(2,\IZ)_q $ bounded strictly between the repelling fixed points of $A$ and $ W_{1/1}A W_{1/1}^{-1} $.
  (When $ \abs{q} > 1 $, the repelling fixed points of $ A$ and $W_{1/1}A W_{1/1}^{-1}$ are $ \infty $ and $1$, respectively, c.f.\ \zcref{tab:eigensystems}.)
\end{rem}

\begin{table}
  \centering
  \caption{Cusp points of denominator at most $8$ on the boundary of the component of $ \mc{Q} $ that lies inside the unit circle. The slopes $ 0/1 $ and $ 1/1 $ correspond to
           the two halves of the Fuchsian locus of the $(2,3)$-Riley slice and are folded together by the $2:1$ map from the Riley slice onto each component of $ \mc{Q} $.\label{tab:cusps}}
  \begin{tabular}{llcll}\toprule
    $r/s$ & cusp &\hspace{1em}&          $r/s$ & cusp\\\midrule
    $0/1$ & $-1.0 + 0.0i$ & &            $1/7$ & $-0.232347 + 0.589797i$\\
    $1/1$ & $-1.0 + 0.0i$ & &            $2/7$ & $0.220025 + 0.507415i$\\
    $1/2$ & $0.381966 + 0.0i$ & &        $3/7$ & $0.477012 + 0.10653i$\\
    $1/3$ & $0.340977 + 0.402816i$ & &   $4/7$ & $0.477012 - 0.10653i$\\
    $2/3$ & $0.340977 - 0.402816i$ & &   $5/7$ & $0.220025 - 0.507415i$\\
    $1/4$ & $0.109976 + 0.519497i$ & &   $6/7$ & $-0.232347 - 0.589797i$\\
    $3/4$ & $0.109976 - 0.519497i$ & &   $1/8$ & $-0.299813 + 0.577406i$\\
    $1/5$ & $-0.0348945 + 0.583025i$ & & $3/8$ & $0.431001 + 0.295843i$\\
    $2/5$ & $0.467623 + 0.214239i$ & &   $5/8$ & $0.431001 - 0.295843i$\\
    $3/5$ & $0.467623 - 0.214239i$ & &   $7/8$ & $-0.299813 - 0.577406i$\\
    $4/5$ & $-0.0348945 - 0.583025i$ & & \\
    $1/6$ & $-0.148582 + 0.578415i$ & & \\
    $5/6$ & $-0.148582 - 0.578415i$ & & \\\bottomrule
  \end{tabular}
\end{table}

\begin{table}
  \centering
  \caption{Eigensystems (hence fixed point dynamics) for $ A $ and $B $.\label{tab:eigensystems}}
  \begin{tabular}{llll}\toprule
    & eigenvalue & eigenvector & behaviour when $ \abs{q} < 1 $\\\midrule
    $A$ & $ (-q)^{1/2} $ & $ (1, q+1)^t $ & repelling\\
     & $ (-q)^{-1/2} $ & $ (1, 0)^t $ & attracting\\
    $B$ & $ (-q)^{1/2} $ & $ (1+q^{-1}, 1)^t $ & repelling\\
     & $ (-q)^{-1/2} $ & $ (0,1)^t $ & attracting\\\bottomrule
  \end{tabular}
\end{table}

\begin{rem}[Behaviour of complex realisations for $ q \in \partial \mc{Q} $]
  The boundary $ \partial \mc{Q} $ can be divided into three sets of points:
  \begin{itemize}
    \item \textit{Cusp points}. These are points $ q $ such that if $ \alpha : [0,1) \to \mc{Q} $ satisfies $ \lim_{t\to 1} \alpha(t) = q $ then
          the conclusion of (2) of \zcref{thm:holomotion} holds but with `fixed points of a conjugate of $A$' replaced with `fixed points of a conjugate of a Farey word in $BA$ and $BAB$'
          (`Farey word' is defined in~\cite{elzenaar26}). The cusp points are dense on the boundary of $ \mc{Q} $. For each cusp point, the map $ \lambda_1 $ sends $ \Lambda(\Gamma(0)) $
          to the complement of a circle packing; i.e.\ the set of realisations of $q$-reals when $q$ is a cusp point is the complement of a circle packing. The value for $q$
          chosen to draw \zcref{fig:quasifuchsian} is slightly away from the $1/3$ cusp point. In \zcref{tab:cusps} we give decimal approximations for the cusp points of slopes with small denominator.
    \item \textit{Degenerate points}. If $ q \in \partial \mc{Q} $ is not a cusp point, then the limit of the quasicircles $ C_t $ as $ t \to 1 $ is a space-filling curve. There are no
          collisions of fixed points, and the map $ \lambda_1 $ is injective. For such $q$, the set of realised $q$-reals is equal to $ \hat{\IC} $. These points are nonmeagre
          in the boundary $ \partial \mc{Q} $ (in the sense of Baire category).
    \item \textit{Punctures}. These are the points $0$ and $ \infty $ and they arise as projections of the fixed point of the parabolic element $ \omega \in \Mod(S_{0,4}) $ such
          that $ \mc{R}_{2,3} = \Teich(S_{0,4})/\langle \omega \rangle $~\cite[\S 4.3]{ems22M}. Since $ \mc{Q} $ is invariant under inversion in the unit circle it suffices to consider
          the groups arising as $ q \to 0 $. Computer experiment suggests that $ \Lambda(\PSL(2,\IZ)_q) \to \{0,1,\infty\} $ (as $ q\to 0$) in the Hausdorff
          metric on compact subsets of the sphere. Here is a sketch of a proof: by our arguments above, for each $ q \in \mc{Q} $ the fixed points of $ W_{r/s} A W_{r/s}^{-1} $
          are dense in the limit set of $ \PSL(2,\IZ)_q $. These are the images of $ \Fix(A) $ under $ W_{r/s} $. Each $ W_{r/s} $ is a product of powers of $ A^{-1} $ and $B$. As $ q \to 0 $,
          both these maps converge uniformly on compact subsets away from their repelling fixed point to the constant map onto their attracting fixed point. Referring to \zcref{tab:eigensystems}, the attracting fixed point of
          $A^{-1}$ is $ 1/(1+q) \to 1 $ and the attracting fixed point of $ B$ is $ 0 $. Thus when $ W_{r/s} \neq 1 $, the fixed points of $W_{r/s} A W_{r/s}^{-1}$ converge to $ \{0,1\} $ as $ q \to \infty $.
          When $ W_{r/s} = 1 $, then the fixed points are $ \infty $ and $ 1/(1+q) \to 1 $. By density, the entire limit set of $ \PSL(2,\IZ)_q $ converges to $ \{0,1,\infty\} $. As a byproduct
          of the argument we see that the only $q$-reals which converge to $ \infty $ as $q\to0$ are $ [\infty]_q $ and the $q$-realisations of the negative reals, since these are the numbers
          for which the corresponding word has leftmost letter $ A^{+1} $.
  \end{itemize}
  See Bers~\cite{bers70}.
\end{rem}

\begin{rem}[Realisations form a quasicircle]
  For $ q \in \mc{Q} $, we see that the set of realisations of $q$-reals lies on a quasicircle (and has measure $0$). This has many analytic consequences---see the survey by Gehring
  and Hag~\cite{gehring}. For instance, the $q$-reals satisfy the Ahfors $3$-point criterion: there exists a constant $ a \geq 1 $ such that for any two $q$-reals
  $ \xi, \zeta \in \IC $, $ \abs{\xi - z} \leq a \abs{\xi-\zeta} $ for any point $ z $ in the arc of $ C \setminus \{\xi, \zeta\} $ with minimum diameter (where $C$ is the Jordan
  curve containing the set of realisations of $q$-reals for the particular $q$). Other possible applications of this point of view include the computation of the Hausdorff
  dimension of the set of realisations of $q$-rationals, see Bishop and Jones~\cite{bishop97} and references therein.
\end{rem}

\begin{rem}[Behaviour of complex realisations for $ q \in \IC \setminus \overline{\mc{Q}} $]
  When $ \PSL(2,\IZ)_q $ is indiscrete, one obtains essentially no information. However, there are many values of $q \in \IC \setminus \overline{\mc{Q}}$ for which $\PSL(2,\IZ)_q$
  is discrete and so the study of `realisations of $q$-reals' can still be done analytically. In \cite[Remark~6.1]{ems24bd} the values of $ q $ on the unit circle for
  which $\PSL(2,\IZ)_q$ is discrete are studied. These groups are Fuchsian triangle groups and so have limit sets behaving in a very similar way to $ \PSL(2,\IZ) $.
  The other values of $q \in \IC \setminus \overline{\mc{Q}}$ corresponding to discrete groups give so-called `generalised Heckoid groups' and are in the process of classification
  by Chesebro, Martin, and Schillewaert~\cite{cms24}.
\end{rem}

Here is a simple geometric corollary. Observe first that if $ q \in \IR_{<0} \inter \mc{Q} $ then $ \PSL(2,\IZ)_q $ is Fuchsian. (The condition that $q$ is negative is to ensure that
the square roots in \eqref{eq:psl2zq} that normalise the matrices to have determinant $1$ are real-valued. It is a sufficient condition but not a necessary one---$ \PSL(2,\IZ)_q $ is
also conjugate to a subgroup of $ \PSL(2,\IR) $ for any $q$ on the unit circle.)
\begin{cor}\label{cor:boundary_hyp}
  For $ q \in \IR $, $ q < -1 $, for each $ r/s \in \IQ $ let $ C[r/s]_q $ denote the circle which is orthogonal to $ \IR $ and passes through $ [r/s]^\sharp_q $ and $ [r/s]^\flat_q $
  Then the circles $ C[r/s]_q $ (the set of `$q$-discs' studied by Jouteur, Paris-Romaskevich, and Thomas~\cite{jouteur26}) are disjoint and no $q$-real lies in their interior.
\end{cor}
\begin{proof}
  The result is trivial for $ \PSL(2,\IZ) $. By \zcref{thm:holomotion}, as we deform $ q $ along $ \IR $ away from $ -1 $, the limit set of $ \PSL(2,\IZ)_q $ is exactly the set
  of $q$-reals and they `pull apart' from each other in an order-preserving way leaving gaps between each pair $ [r/s]^\sharp_q $ and $ [r/s]^\sharp_q $
  (since these are just the fixed points of $ W_{r/s} A W_{r/s}^{-1} $).
\end{proof}

While we are discussing the Fuchsian setting, we observe that the Ford domains of $ \PSL(2,\IZ)_q $ are relatively easy to compute in this case.
In fact, fundamental domains for all Fuchsian groups on two elliptic or parabolic generators (as well as inequalities on the traces of the generators
of arbitrary pairs of parabolic or elliptic matrices with real coefficients which characterise when they generate a Fuchsian group) were given by
Knapp~\cite{knapp68} and Rosenberger~\cite{rosenberger86}. These domains were rediscovered in the context of $q$-rationals by Jouteur, Paris-Romaskevich,
and Thomas~\cite[Proposition~A]{jouteur26}.

\begin{figure}
  \centering
  \includegraphics[width=.7\textwidth]{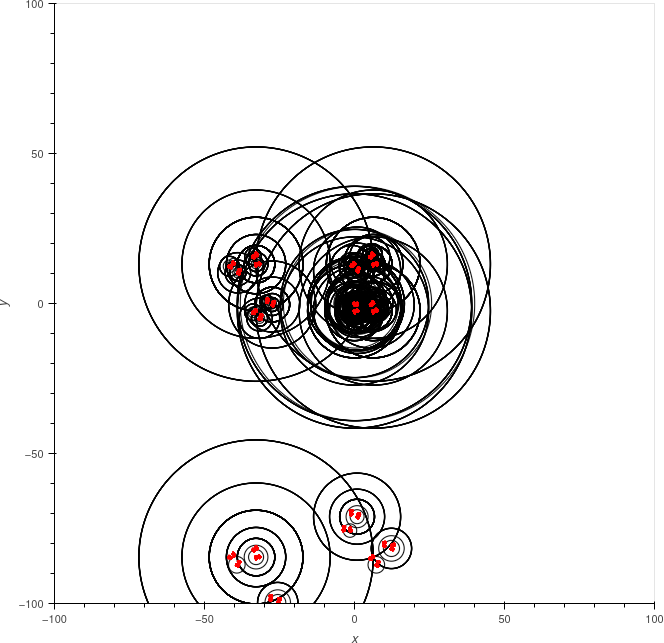}
  \caption{Isometric circles for words of length $ \leq 8 $ in $ \PSL(2,\IZ)_q $ where $ q = 0.4i $.\label{fig:quasifuchsian_ford}}
\end{figure}

\begin{rem}[Ford domains for arbitrary $ q \in \mc{Q} $]
  One can attempt to study fundamental domains of $ \PSL(2,\IZ)_q $ for $ q \in \mc{Q} \setminus \IR $. In general, one can always find finite-sided fundamental domains for $ \PSL(2,\IQ)_q $.
  However, as $q$ moves closer and closer to $ \partial \mc{Q} $ away from $ \IR $ the number of faces will increase without bound. One can study the structure of the Ford domains via
  essentially the same machinery as that described by Akiyoshi, Sakuma, Wada, and Yamashita~\cite{akiyoshi} following J\o{}rgensen~\cite{jorgensen03} (there is a cell
  decomposition of $ \mc{Q} $ indexed by the Farey tree such that (i) if $ q $ lies in a given cell then the Ford domain arises from the isometric circles of some fixed set of elements,
  and (ii) if $q$ moves from one cell to a neighbouring cell then the isometric circles of elements corresponding to that edge need to be added or removed). One can also produce Ford domains for fixed $q$
  computationally as studied by Riley~\cite{riley83}. In \zcref{fig:quasifuchsian_ford} we show the isometric circles for one such group.
\end{rem}

In this direction, we can also reprove Proposition~2.14 of Bapat, Becker, and Licata~\cite{bapat22} and prove Conjecture 2.15 of that paper. We must first
recall the following definition from Morier-Genoid and Ovsienko~\cite[\S 2]{moriergenoid20} and Bapat, Becker, and Licata~\cite[\S 2.2]{bapat22}.
\begin{defn}\label{defn:farey_tess}
  The \df{$q$-deformed Farey tessellation} is the simplicial complex in $ \IH^2 $ obtained from the images of the hyperbolic triangle
  \begin{displaymath}
    \hconv \{ 0/1, 1/0, 1/1 \}
  \end{displaymath}
  under the rational words.
  The resulting triangles come equipped with an edge weighting and a recursion akin to Farey addition which depends on this weight~\cite[Figure~2]{bapat22}
  but we do not require this additional structure here. The vertices of the $q$-deformed Farey tessellation are the right $q$-deformed rationals.
\end{defn}
Our definition appears slightly different to that in \cite{bapat22}, as they define the tessellation via a recursion on the edges coming from the Farey rule.
This is equivalent by the discussion in the previous section: the actions of $A^{-1}$ and $B$ on a triangle produce the two neighbouring triangles whose vertices are given
by Farey addition, and then $ S = ABA $ reflects the whole tessellation to give the negative $q$-deformed triangles.

\begin{cor}\label{cor:shape_of_qfareytess}
  For $ q \in \IR $, $ q < -1 $, let $ L_q \subset \IH^2 $ be the proper subset of $ \IH^2 $ covered by the $q$-deformed Farey tessellation. Then:
  \begin{enumerate}
    \item The set $ L_q $ is homeomorphic to an open disc, and its closure is homeomorphic to a closed disc.
    \item The boundary $ \partial \overline{L_q} $ is the union of the $q$-deformed rational curves and the $q$-deformed irrationals.
    \item There is a homeomorphism $ \partial \overline{L_q} \to \IR \union \{\infty\} $ which restricts to the identity on the left and right $q$-rationals and the $q$-deformed irrationals.
    \item The set of $q$-deformed rational curves is dense in $ \partial \overline{L_q} $.
  \end{enumerate}
\end{cor}
\begin{proof}
  When $q$ is in the range $ (-\infty,-1) $, it lies in the interior of $ \mc{Q} $ and the group $ \PSL(2,\IZ)_q$ is a Fuchsian group acting separately on the upper and lower halfplanes,
  projecting each down to a disc with two marked points of orders $2$ and $3$. Further, the element $A$ is hyperbolic with some axis $\Ax(A)$ in $ \IH^2 $, and in fact is a primitive
  boundary hyperbolic as defined in~\cite[\S 10.3]{beardon}. Consider the fundamental quadrilateral $Q = \hconv\{0,1,\infty\}$ for the Farey
  tessellation defined in \zcref{defn:farey_tess}. The two ends of $ \Ax(A) $ are $ \infty $ and $ 1/(1+q) $, and the latter is an attracting fixed point (\zcref{tab:eigensystems}). Thus the
  quadrilaterals $ A^n Q $ limit onto $ \Ax(A) $, so $\overline{L_q}$ contains $ \Ax(A) $.

  Let $ gAg^{-1} $ be some other boundary hyperbolic (so $ g \in \PSL(2,\IZ)_q$ is arbitrary), so the axis $ \Ax(gAg^{-1}) $ joins the $g$-translates
  of the fixed points of $A$. By \zcref{thm:holomotion}, as $ q \to -1 $ along the ray $ (-\infty,-1) $ these two fixed points converge to a single parabolic fixed point of $ \PSL(2,\IZ) $.
  Thus they are a pair of left and right $q$-rationals and so there exists some $ r/s \in \IQ $ such that $ \Ax (gAg^{-1}) = W_{r/s} \Ax(A) $. By the same argument to the first
  paragraph, but now taking powers of $ W_{r/s} A W_{r/s}^{-1} $, $ W_{r/s} \Ax(A) $ also lies in $\overline{L_q}$. Hence every boundary
  hyperbolic axis of $ \PSL(2,\IZ)_q$ lies in $\overline{L_q}$.

  We now observe that $ \overline{L_q} $ contains every $q$-irrational. Indeed, every $q$-irrational $ \xi $ is an accumulation point of the sequence $ W_{\alpha_i} Q $
  for rational approximants $ \alpha_i \to \xi $. This implies that $\overline{L_q}$ is a convex set (any union of ideal triangles in $ \IH^2 $ is convex) which
  meets every limit point of $ \PSL(2,\IZ)_q $, and is bounded by the
  set of boundary hyperbolic axes. It is therefore the hyperbolic convex hull of the limit set of $ \PSL(2,\IZ)_q $, and it is well-known that this is a homeomorphic image of $ \overline{\IH^2} $, in fact
  a contraction (to construct the homeomorphism for some Fuchsian group $ G $, modify the nearest point retraction $ \Omega(G) \inter \overline{\IH^2} \to \hconv \Lambda(G) $ by pushing parts of $\hconv \Lambda(G)$
  near faces inside $ \IH^2 $ further into the interior).
\end{proof}

We can deform $ q $ into $ \mc{Q} \setminus \IR $ and continue to look for realisations for the Farey tessellation. There are a couple of options. First, one could take the convex
hull in $ \IH^3 $ rather than $ \IH^2 $. This will carry additional combinatorial data that is not related to the Farey tessellation (namely, it will be `folded' according to the position
of $q \in \mc{Q} $)~\cite[Chapter~8]{thurstonN}. Alternatively, one could take the convex hull with respect to the Poincar\'e metric~\cite[Theorem 1-9]{ahlfors} on one of the quasidiscs
preserved by $ \PSL(2,\IZ)_q $---but this will not have circular arcs as boundary curves, just analytic curves.

\sloppy\printbibliography

\end{document}